\documentclass[preprint,12pt, leqno]{elsarticle}

\usepackage{graphicx, color}

\usepackage{amssymb}
 \usepackage{amsthm}

\usepackage{amsmath} 
\usepackage{amscd}
\newtheorem{thm}{Theorem}[section]
\newtheorem{dfn}[thm]{Definition}
\newtheorem{lmm}[thm]{Lemma}
\newtheorem{crl}[thm]{Corollary}
\newtheorem{rem}[thm]{Remark}
\numberwithin{equation}{section}

\usepackage[all]{xy}

\newcommand{\Alg}{\mbox{{\rm Alg}}}

\newcommand{\Z}{\Bbb Z}

\newcommand{\R}{\Bbb R}

\newcommand{\RP}{\Bbb R\mbox{{\rm P}}}

\newcommand{\Map}{\mbox{{\rm Map}}}

\newcommand{\CP}{\Bbb C {\rm P}}

\newcommand{\dis}{\displaystyle}
\newcommand{\p}{\prime}

\newcommand{\E}{\tilde{E}}
\newcommand{\XD}{X^{\Delta}}
\newcommand{\SZ}{{\mathcal{X}}^{\Delta ,d}}
\newcommand{\SZd}{{\mathcal{X}}^{\Delta ,d+2}}

\newcommand{\I}{\mbox{{\rm (i)}}}
\newcommand{\II}{\mbox{{\rm (ii)}}}
\newcommand{\III}{\mbox{{\rm (iii)}}}

\journal{Topology and its Applications}

\begin{document}

\begin{frontmatter}




\title{\bf Simplicial resolutions and spaces of algebraic maps between real
projective spaces}



\author[label1]{A. Kozlowski\corref{cor1}}
\ead{akoz@mimuw.ed.pl}
\address[label1]{Institute of Applied Mathematics and Mechanics,University of Warsaw,
Banacha 2, 02-097 Warsaw,
Poland}
\author[label2]{K. Yamaguchi}
\ead{kohhei@im.uec.ac.jp}
\address[label2]{Department of Mathematics, University of Electro-Communications,
1-5-1 Chofugaoka, Chofu, Tokyo 182-8585, Japan}
\cortext[cor1]{Corresponding author}

\begin{abstract}
We show that the space $\tilde{A}_{d}(m,n)$  consisting
of all real projective classes of $(n+1)$-tuples of 
real coefficients homogeneous
 polynomials of degree $d$ in 
$(m+1)$ variables, without common real roots except zero, has the same homology as the space $ \Map(\RP^m,\Bbb  \RP^n)$ of continuous maps from the $m$-dimensional real projective space $\RP^m$ into
the $n$ real dimensional projective space $\RP^n$
up to dimension 
$(n-m)(d+1)-1$.   This  considerably improves the main result of \cite{AKY1}.
\end{abstract}


\end{frontmatter}


\section{Introduction.}
Let $M$ and $N$ be manifolds with some additional structure, e.g holomorphic, symplectic, real algebraic etc. The relation between the topology of the space of continuous maps preserving this structure and that of the space of all continuous maps has long been an object of study in several areas of topology and geometry (e.g. \cite{BHM}, \cite{CJS} 
\cite{GKY2}, \cite{KY1}, \cite{KY3}, \cite{Mo2}, \cite{Se}). 
In  \cite{AKY1} we considered the case where the structure is that of a real algebraic variety.  The continuous maps that preserve this structure are the rational maps 
(called also regular or real algebraic maps).  
Using a method analogous to the one invented by Mostovoy \cite{Mo2} 
 to deal with the complex case, we obtained a closely related result, in which the space of rational maps is replaced by the space of tuples of homogeneous polynomials without non-trivial roots. 
 (In fact we believe this space to be equivalent to the space of all rational maps but have not been able to prove this completely). 
 Mostovoy's original article contained some mistakes that were corrected in \cite{Mo3}. 
 Besides correcting of  mistakes, \cite{Mo3} improves the  results of \cite{Mo2} 
 by means of a new variant of the  main spectral sequence. 
 This new spectral sequence is obtained from a \lq\lq truncated simplicial resolution\rq\rq\  of a discriminant, which replaces the \lq\lq degenerate simplicial resolution\rq\rq\  used in the original argument. 
 It is easy to see that the new spectral sequence can also be applied 
 to our case and makes it possible to obtain  better results. 
 The main purpose of this article is to carry out this idea.  
 However, we are able to obtain  better results than could be derived by a straight-forward replacement of the degenerate resolutions used in \cite{AKY1} by truncated ones. 
 This is achieved by  combining the information obtained 
 by using truncated resolutions (Theorem \ref{thm: sd} below) 
 with information obtained by using a non-degenerate resolution in the manner of \cite{Va}. 
We conjecture that our main result (Theorem \ref{thm: KY4-I} below) is optimal. 
\paragraph{\bf Notation}
Let $m$ and $n$ be positive integers such that
$1\leq m <n$.
We choose  ${\bf e}_k=[1:0:\cdots :0]\in \RP^k$ 
as the base point of $\RP^k$ ($k=m,n$), and
we denote by $\Map^*(\RP^m,\Bbb  \RP^n)$ the space
consisting of all based maps
$f:(\RP^m,{\bf e}_{m})\to (\Bbb \RP^n,{\bf e}_n)$.
We also denote
by $\Map_{\epsilon}^* (\RP^m,\RP^n)$
the corresponding path component
of $\Map^* (\RP^m,\RP^n)$
for each
$\epsilon \in \Z/2=\{0,1\}=
\pi_0(\Map^*(\RP^m,\Bbb  \RP^n))$
(\cite{CS}). 
Similarly, let $\Map (\RP^m,\RP^n)$ denote the space
of all free maps $f:\RP^m\to\RP^n$ and
$\Map_{\epsilon}(\RP^m,\RP^n)$ the corresponding path component
of $\Map (\RP^m,\RP^n)$.
\par
We shall use the symbols $z_i$ to denote variables of polynomials.
A  map $f:\RP^m\to \RP^n$ is called a
{\it algebraic map of the degree }$d$ if it can be represented as
a rational map of the form
$f=[f_0:\cdots :f_n]$ such that
 $f_0,\cdots ,f_n\in\Bbb R [z_0,\cdots ,z_m]$ are homogeneous polynomials of the same degree $d$
with no common {\it real } roots except ${\bf 0}_{m+1}=(0,\cdots ,0)\in\R^{m+1}$.
 We  denote by $\Alg_d(\RP^m,\RP^n)$ (resp. $\Alg_d^*(\RP^m,\RP^n)$) the space
 consisting of all (resp. based) algebraic maps $f:\RP^m\to \RP^n$
 of degree $d$.
 It is easy to see that there are inclusions 
 $\Alg_d(\RP^m,\RP^n)\subset \Map_{[d]_2}(\RP^m,\RP^n)$
 and
 $\Alg^*_d(\RP^m,\RP^n)\subset \Map^*_{[d]_2}(\RP^m,\RP^n)$,
 where $[d]_2\in\Z/2=\{0,1\}$ denotes the integer $d$ mod $2$.
\par  
For $m\geq 2$ and given an algebraic map 
$g\in \Alg_{d}^*(\RP^{m-1},\RP^n)$,
we denote by $\Alg_d(m,n;g)$ and $F(m,n;g)$  the subspaces given by
$$
\begin{cases}
\Alg_d(m,n;g) &=\ 
\{f\in \Alg_d^*(\RP^m,\RP^n):f\vert \RP^{m-1}=g\},
\\
F(m,n;g)  &=\ \{f\in \Map_{[d]_2}^*(\RP^m,\RP^n):
f\vert \RP^{m-1}=g\}.
\end{cases}
$$
Note that there is a homotopy equivalence
$F(m,n;g)\simeq  \Omega^m\RP^n \simeq \Omega^mS^n$ (\cite{Sasao}).

\paragraph{\bf Spaces of polynomials}
Now we turn to  spaces of collections of polynomials which represent  algebraic maps. These will all be subsets of real affine or real projective spaces.
\par\vspace{2mm}\par
Let ${\cal H}_{d,m}\subset \R [z_0,\ldots ,z_m]$
denote the subspace consisting of all
homogeneous polynomials of degree $d$.
Let $ A_{d}(m,n)(\Bbb R)\subset {\cal H}_{d,m}^{n+1}$ denote the space of all $(n+1)$-tuples 
$(f_0,\ldots ,f_n)\in \R[z_0,\ldots ,z_m]^{n+1}$
of homogeneous polynomials of degree $d$ without non-trivial common real roots
(but possibly with non-trivial common {\it non-real} roots). These are precisely the collections that represent algebraic maps. Since the algebraic map is invariant under multiplication of all polynomials by a non-zero scalar, it is convenient to define the projectivisation 
$\tilde{A}_d(m,n)= A_{d}(m,n)(\Bbb R) / \R^*,$
which is a subset of the real projective space
$({\cal H}_{d,m}^{n+1}\setminus\{{\bf 0}\})/\R^*$.
We have a projection
$
\Gamma_d:\tilde{A}_d(m,n)\to \Alg_d(\RP^m,\RP^n).
$
%
\par\vspace{2mm}\par
To describe base-point preserving maps we need to use the subspace 
$A_{d}(m,n)\subset A_d(m,n)(\R)$, which consists of $(n+1)$-tuples $(f_0,\ldots ,f_n)\in A_d(m,n)(\R)$ such that the coefficient at $z_0^d$ is $1$ in $f_0$ and $0$ in the remaining $f_k$. Note that every algebraic map which preserves the base-point has a representation by such a collection of polynomials, and
we get a natural projection map
$
\Psi_d:A_d(m,n)\to \Alg_d^*(\RP^m,\RP^n).
$
%
\par\vspace{2mm}\par
We define polynomial representations of restricted maps.
From now on,
let $m\geq 2$ and $g\in \Alg_d^*(\RP^{m-1},\RP^n)$ be a based algebraic map 
with some fixed polynomial representation $g=[g_0:\cdots :g_n]$ such that 
$(g_0,\ldots ,g_n)\in A_d(m-1,n)$. Set $B_k=\{g_k+z_mh:h\in {\cal H}_{d-1,m}\}$ $(k=0,1,\ldots ,n)$ and let
\begin{equation}\label{ad}
A_d^*:=B_0\times B_1\times \cdots \times B_n\subset  {\cal H}_{d,m}^{n+1}.
\end{equation}
This space consists of collections of polynomials which restrict to $(g_0,\ldots ,g_n)$ when $z_m=0$. Define  $A_d(m,n;g)\subset A_d^*$ to be the subspace of all collections with 
no non-trivial common {\it real} root:
\begin{equation}\label{Ad}
A_d(m,n;g)=
\left\{
\begin{array}{c}
\mbox{ }
\\
(f_0,\cdots ,f_n)\in A_d^*
\\
\mbox{ }
\end{array}
\right|
\left.
\begin{array}{l}
f_0,\cdots ,f_n
\mbox{ have no }
\\
\mbox{common real root}
\\
\mbox{except }
{\bf 0}_{m+1}
\in\R^{m+1}
\end{array}
\right\}.
\end{equation}
Clearly, such a collection determines a map in $\Alg_d(m,n;g)$ and we again obtain a projection
$\Psi_d^{\p}:A_d(m,n;g)\to \Alg_d(m,n;g).$
Let 
\begin{equation}
\begin{cases}
j_{d,\R}:\Alg_d(\RP^m,\RP^n)\stackrel{\subset}{\rightarrow} \Map_{[d]_2}(\RP^m,\RP^n)
\\
i_{d,\R}:\Alg_d^*(\RP^m,\RP^n)\stackrel{\subset}{\rightarrow} \Map_{[d]_2}^*(\RP^m,\RP^n)
\\
i_{d,\R}^{\p}:\Alg_d(m,n;g)\stackrel{\subset}{\rightarrow} F(m,n;g)\simeq \Omega^m\RP^n
\simeq \Omega^mS^n
\end{cases}
\end{equation}
denote the inclusions and let
\begin{equation}\label{jd}
\begin{cases}
j_d=j_{d,\R}\circ \Gamma_d:\tilde{A}_d(m,n)\to \Map_{[d]_2}^*(\RP^m,\RP^n)
\\
i_d=i_{d,\R}\circ \Psi_d:A_d(m,n)\to \Map_{[d]_2}^*(\RP^m,\RP^n)
\\
i_d^{\p}=i_{d,\R}^{\p}\circ \Psi_d^{\p}:A_d(m,n;g)\to F(m,n;g)\simeq \Omega^mS^n
\end{cases}
\end{equation}
be the natural maps.
\par\vspace{2mm}\par

The contents of this section can be summarized in the following 
diagram, for which we assume $g$ is a based algebraic map of degree $d$.
\[
\xymatrix{%
\Map_{[d]_2}(\RP^m,\RP^n)     &     \Map^*_{[d]_2}(\RP^m,\RP^n) \ar@{_{(}->}[l]  &   F(m,n;g) \ar@{_{(}->}[l] \\
\Alg_d(\RP^m,\RP^n) \ar@{^{(}->}[u]^{j_{d,\R}}     &     \Alg_d^*(\RP^m,\RP^n) \ar@{^{(}->}[u]^{i_{d,\R}} \ar@{_{(}->}[l]  &   \Alg_d(m,n;g)  \ar@{^{(}->}[u]^{i_{d,\R}^{\p}} \ar@{_{(}->}[l]\\
\tilde{A}_d(m,n) \ar@{->>}[u]^{\Gamma_d} \ar@/^4pc/[uu]^>>>>>>{j_d}   &     A_d(m,n) \ar@{->>}[u]^{\Psi_d} \ar@/^4pc/[uu]^>>>>>>{i_d} &   A_d(m,n;g) \ar@{->>}[u]^{\Psi_d^{\p}} \ar@/^4pc/[uu]^>>>>>>{i_d^{\p}}
}
\]
\paragraph{\bf The main results}
Let $\lfloor x\rfloor$ denote the integer part of a real number $x$, 
and define the integers
$D_{*}(d;m,n)$
and $D(d;m,n)$
 by
\begin{equation}\label{Dnumber}
D_{*}(d;m,n)  =
(n-m)\big(\lfloor \frac{d+1}{2}\rfloor  +1\big) -1,
\ \
D(d;m,n) =
(n-m)(d+1) -1.
\end{equation}
Now, we recall 
the following 2  results.

%
\begin{thm}[\cite{KY1}, \cite{Mo1}, \cite{Y5}]\label{thm: A1}
If $n\geq 2$ is an integer and $m=1$, then
the natural map
$i_{d}:A_d(1,n)\to  \Map^*_{[d]_2}(\RP^1,\RP^n)\simeq \Omega S^{n}$
is a homotopy equivalence
up to dimension 
$D(d;1,n)=(n-1)(d+1)-1.$
\qed
\end{thm}
\begin{thm}[\cite{AKY1}]
\label{thm: AKY1-I}
If $2\leq m <n$ are integers and
$g\in \Alg_d^*(\RP^{m-1},\RP^n)$ is a fixed based algebraic map of degree $d$,
the natural map
$i_{d}^{\p}:A_d(m,n;g)\to F(m,n;g)\simeq \Omega^mS^n$
is a homotopy equivalence through dimension $D_{*}(d;m,n)$ if $m+2\leq n$
and a homology equivalence through dimension $D_{*}(d;m,n)$ if $m+1=n$.
\qed
\end{thm}
\begin{thm}[\cite{AKY1}]
\label{thm: AKY1-II}
If $\ 2\leq m <n$ are integers, the natural maps
$$
\begin{cases}
j_d:\tilde{A}_d(m,n)\to \Map_{[d]_2}(\RP^m,\RP^n)
\\
i_d:A_d(m,n)\to \Map_{[d]_2}^*(\RP^m,\RP^n)
\end{cases}
$$
are homotopy equivalences through dimension $D_{*}(d;m,n)$ if
$m+2\leq n$ and  homology equivalences through dimension
$D_{*}(d;m,n)$ if $m+1=n$.
\qed
\end{thm}
\begin{rem}\label{Remark: ho}
{\rm
A map $f:X\to Y$ is called {\it a homotopy} (resp. {\it a homology}) 
{\it equivalence through dimension} $D$ if
$f_*:\pi_k(X)\to \pi_k(Y)$ (resp. $f_*:H_k(X,\Z)\to H_k(Y,\Z)$) is an isomorphism for any
$k\leq D$.
Similarly, it is called {\it a homotopy} (resp. {\it a homology}) {\it equivalence up to dimension} $D$ if
$f_*:\pi_k(X)\to \pi_k(Y)$ (resp. $ f_*:H_k(X,\Z)\to H_k(Y,\Z)$) is an isomorphism for any $k< D$ and an epimorphism for $k=D$.
}
\end{rem}
\par\vspace{2mm}\par
The main purpose of this paper is to improve the above results by replacing $D_{*}(d;m,n)$ by $D_{}(d;m,n)$.  
More precisely, we will prove the following.

\begin{thm}
\label{thm: KY4-I}
If $2\leq m <n$ are integers and
$g\in \Alg_d^*(\RP^{m-1},\RP^n)$ is a fixed based algebraic map of degree $d$,
the natural map
$i_{d}^{\p}:A_d(m,n;g)\to F(m,n;g)\simeq \Omega^mS^n$
is a homotopy equivalence up to dimension $D_{}(d;m,n)$ if $m+2\leq n$
and a homology equivalence up to dimension $D_{}(d;m,n)$ if $m+1=n$.
%
\end{thm}
Using the same method as in \cite{AKY1}
(cf. \cite{GM}), we also obtain the following:

\begin{crl}
\label{thm: KY4-II}
If $\ 2\leq m<n$ are integers, 
the natural maps
$$
\begin{cases}
j_d:\tilde{A}_d(m,n)\to \Map_{[d]_2}(\RP^m,\RP^n)
\\
i_d:A_d(m,n)\to \Map_{[d]_2}^*(\RP^m,\RP^n)
\end{cases}
$$
are homotopy equivalences 
up to 
dimension 
$D_{}(d;m,n)$ if
$m+2\leq n$ and  homology equivalences 
up to 
dimension
$D_{}(d;m,n)$ if $m+1=n$.
\qed
\end{crl}
\begin{rem}
{\rm
With the help of Theorem \ref{thm: KY4-I}, it is possible to
generalize some results given in \cite{KY3}
concerning the space of algebraic maps
from $\RP^m$ to $\CP^n$. We do this in
a separate paper \cite{KY5}.
}
\end{rem}
\par\vspace{2mm}\par
This paper is organized as follows.
In section 2, we discuss various kinds of simplicial resolutions and describe a key new idea of Mostovoy \cite{Mo2}.
In section 3 we study the spectral sequences induced from the
non-degenerate simplicial resolutions of our spaces of discriminants. 
In section 4, we introduce the spectral sequence induced from the truncated simplicial resolution of our spaces of discriminants and prove the key
Theorem (Theorem \ref{thm: sd}).
Finally, in section 5, we prove the
 main result (Theorem \ref{thm: KY4-I}) by using the spectral sequences
induced from the non-degenerate simplicial resolutions of stabilized discriminants. 

\section{Simplicial resolutions.}\label{section 2}

In this section, we shall recall the definition of non-singular simplicial resolutions
and their truncated simplicial resolutions.
\begin{dfn}
{\rm
(i) For a finite set ${\bf x} =\{x_1,\cdots ,x_l\}\subset \R^N$,
let $\sigma (\textbf{x})$ denote the convex hull spanned by ${\bf x}.$
Let $h:X\to Y$ be a surjective map such that
$h^{-1}(y)$ is a finite set for any $y\in Y$, and let
$i:X\to \R^n$ be an embedding.
\par
Let  $\mathcal{X}^{\Delta}$  and $h^{\Delta}:{\mathcal{X}}^{\Delta}\to Y$ 
denote the space and the map
defined by
$$
\mathcal{X}^{\Delta}=
\big\{(y,u)\in Y\times \R^N:
u\in \sigma (i(h^{-1}(y)))
\big\}\subset Y\times \R^N,
\ h^{\Delta}(y,u)=y.
$$
The pair $(\mathcal{X}^{\Delta},h^{\Delta})$ is called
{\it a simplicial resolution of }$(h,i)$.
In particular, $(\mathcal{X}^{\Delta},h^{\Delta})$
is called {\it a non-degenerate simplicial resolution} if for each $y\in Y$
any $k$ points of $i(h^{-1}(y))$ span $(k-1)$-dimensional simplex of $\R^N$.
\par
(ii)
For each $r\geq 0$, let $\mathcal{X}^{\Delta}_r\subset \mathcal{X}^{\Delta}$ be the subspace
given by 
$$
\mathcal{X}_r^{\Delta}=\big\{(y,\omega)\in \mathcal{X}^{\Delta}:
\omega\in\sigma ({\bf v}),
{\bf v}=\{v_1,\cdots ,v_l\}\subset i(h^{-1}(y)),l\leq r\big\}.
$$
We make identification $X=\mathcal{X}^{\Delta}_1$ by identifying 
 $x\in X$ with 
$(h(x),i(x))\in \mathcal{X}^{\Delta}_1$,
and we note that  there is an increasing filtration
$$
\emptyset =
\mathcal{X}^{\Delta}_0\subset X=\mathcal{X}^{\Delta}_1\subset \mathcal{X}^{\Delta}_2\subset
\cdots \subset \mathcal{X}^{\Delta}_r\subset \mathcal{X}^{\Delta}_{r+1}\subset
\cdots \subset \bigcup_{r= 0}^{\infty}\mathcal{X}^{\Delta}_r=\mathcal{X}^{\Delta}.
$$
}
\end{dfn}

\begin{lmm}[\cite{Mo2}, \cite{Mo3}, \cite{Va}]\label{lemma: simp}
Let $h:X\to Y$ be a surjective map such that
$h^{-1}(y)$ is a finite set for any $y\in Y,$ and let
$i:X\to \R^N$ be an embedding.
\par
$\I$
If $X$ and $Y$ are semi-algebraic spaces and the
two maps $h$, $i$ are semi-algebraic maps, then
$h^{\Delta}:\mathcal{X}^{\Delta}\stackrel{\simeq}{\rightarrow}Y$
is a homotopy equivalence.
\par
$\II$
There is an embedding $j:X\to \R^M$ such that the associated simplicial resolution
$(\tilde{\mathcal{X}}^{\Delta},\tilde{h}^{\Delta})$
of $(h,j)$ is non-degenerate, and
the space $\tilde{\mathcal{X}}^{\Delta}$
is uniquely determined up to homeomorphism.
Moreover,
there is a filtration preserving homotopy equivalence
$q^{\Delta}:\tilde{\mathcal{X}}^{\Delta}\stackrel{\simeq}{\rightarrow}{\mathcal{X}}^{\Delta}$ such that $q^{\Delta}\vert X=\mbox{id}_X$.
\qed
\end{lmm}

\begin{rem}\label{Remark: non-degenerate}
{\rm
Even for a  surjective map $h:X\to Y$ which is not finite to one,  
it is still possible to construct an associated non-degenerate simplicial resolution.
In fact, a non-degenerate simplicial resolution may be
constructed by choosing a sequence of embeddings
$\{\tilde{i}_r:X\to \R^{N_r}\}_{r\geq 1}$ satisfying the following two conditions
for each $r\geq 1$ (cf. \cite{Va}).
\begin{enumerate}
\item[(\ref{lemma: simp}$)_r$]
\begin{enumerate}
\item[(i)]
For any $y\in Y$,
any $t$ points of the set $\tilde{i}_k(h^{-1}(y))$ span $(t-1)$-dimensional affine subspace
of $\R^{N_r}$ if $t\leq 2r$.
\item[(ii)]
$N_r\leq N_{r+1}$ and if we identify $\R^{N_r}$ with a subspace of
$\R^{N_{r+1}}$, 
then $\tilde{i}_{r+1}=\hat{i}\circ \tilde{i}_r$,
where
$\hat{i}:\R^{N_r}\stackrel{\subset}{\rightarrow} \R^{N_{r+1}}$
denotes the inclusion.
\end{enumerate}
\end{enumerate}
Let
$$
\dis\mathcal{X}^{\Delta}_r=\big\{(y,\omega)\in Y\times \R^{N_r}:
\omega\in\sigma ({\bf u}),
{\bf u}=\{u_1,\cdots ,u_l\}\subset \tilde{i}_k(h^{-1}(y)),l\leq r\big\}.
$$
Then
by identifying naturally  ${\cal X}^{\Delta}_r$ with a subspace
of ${\cal X}_{r+1}^{\Delta}$,  we now define the non-degenerate simplicial
resolution ${\cal X}^{\Delta}$ of  $h$ as the union  
$\dis {\cal X}^{\Delta}=\bigcup_{r\geq 1} {\cal X}^{\Delta}_r$.
Non-degenerate simplicial resolutions have a long been used in algebraic geometry, and play  the central role in the work of Vassiliev  \cite{Va}.}
\end{rem}

\par\vspace{2mm}\par

In many practical cases the embedding  used to construct a simplicial resolution is given by an explicit map which carries geometric information about the corresponding filtration. Typically such an embedding gives rise to a simplicial resolution that is non-degenerate only in low dimensions. In some situations, such a degenerate resolution may provide more information about the homology of the resolved space than the non-degenerate one. This is why a degenerate resolution (defined by a Veronese-like embedding defined in the next section)  was used in \cite{Mo2} and \cite{AKY1}.  However, in \cite{Mo3} a modification of the non-degenerate resolution, called, the truncated (after a a certain term)
simplicial resolution  was used to obtain results that are (in most dimensions) better than the one derived by means of the degenerate resolution.

\begin{dfn}
{\rm
Let $h:X\to Y$ be a surjective semi-algebraic map between semi-algebraic spaces, and
$j:X\to \R^N$ be a semi-algebraic embedding. Consider the 
 associated non-degenerate  simplicial resolution
$({\cal X}^{\Delta},h^{\Delta}:{\cal X}^{\Delta}\to Y)$
\par
Let $k$ be a fixed positive integer and let
$h_k:{\cal X}^{\Delta}\to Y$ be the map defined
defined by the restriction
$h_k:=h^{\Delta}\vert {\cal X}^{\Delta}_k$.
The fibres of the map $h_k$ are $(k-1)$-skeleta of the fibres of $h^{\Delta}$ and, in general,  fail to be simplices over the subspace
$$
Y_k=\{y\in Y:h^{-1}(y)\mbox{ consists of more than $k$ points}\}.
$$
Let $Y(k)$ denote the closure of the subspace $Y_k$.
We modify the subspace ${\cal X}^{\Delta}_k$ so as to make the all
the fibres of $h_k$ contractible by adding to each fibre of $Y(k)$ a cone whose base
is this fibre.
We denote by $X^{\Delta}(k)$ this resulting space and by
$h^{\Delta}_k:X^{\Delta}(k)\to Y$ the natural extension of $h_k$.
\par
Following Mostovoy \cite{Mo3} we call the map $h^{\Delta}_k:X^{\Delta}(k)\to Y$
{\it the truncated (after the $k$-th term)  simplicial resolution} of $Y$ .
Note that there is a natural filtration
$$
\emptyset =X^{\Delta}_0\subset X^{\Delta}_1\subset
X^{\Delta}_2\subset
\cdots 
\subset X^{\Delta}_{k-1}\subset X^{\Delta}_k\subset X^{\Delta}_{k+1}=
X^{\Delta}_{k+2}
=\cdots =X^{\Delta}(k),
$$
where $X^{\Delta}_r={\cal X}^{\Delta}_r$ if $r\leq k$ and
$X^{\Delta}_r=X^{\Delta}(k)$ if $r>k$.
}
\end{dfn}
\begin{rem}\label{Remark 4}
{\rm
Let $\varphi:A\to B$ be a simplicial map between simplicial complexes and
let $C_{f}\varphi :C_{f}A\to B$ be its fibrewise cone construction.
Because $C_f\varphi$ can be constructed in the category of simplicial complexes,
it is a quasi-fibration and so it is a homotopy equivalence \cite{Hatch}.
It is also know that
for a semi-algebraic map $f:X\to Y$ between semi-algebraic spaces, there are
semi-algebraic triangulations on $X$ and $Y$ such that 
there is a semi-algebraic trivialization of the map $f$
\cite[Theorems 9.2.1 and 9.3.2]{BCR}.
}
\end{rem}
Then using these results, it is not difficult to prove the following.

\begin{lmm}[\cite{Mo3}]\label{Lemma: truncated}
Let $h:X\to Y$ be a surjective semi-algebraic map between semi-algebraic spaces, 
let $j:X\to \R^n$ be a semi-algebraic embedding with the associated simplicial resolution
$({\cal X}^{\Delta},h^{\Delta}:{\cal X}^{\Delta}\to Y)$, and let
$h^{\Delta}_k:X^{\Delta}(k)\to Y$
denote
the corresponding truncated $($after $k$-th term$)$ simplicial resolution of $Y$.
Then
$h^{\Delta}_k$ is a quasi-fibration with contractible fibres and hence a homotopy equivalence.
\qed
\end{lmm}

The truncated simplicial resolution has two properties that play a key role in the new, improved argument. First of all,  as we saw above, the natural filtration stabilizes above the truncation degree. The second key property is the following:

\begin{lmm}[\cite{Mo3}, Lemma 2.1]\label{Lemma: dimension}
Under the same assumptions as Lemma \ref{Lemma: truncated},
$$
\dim (X_{k+1}^{\Delta}\setminus X_k^{\Delta}) = 
\dim ({\cal X}_k^{\Delta} \setminus {\cal X}_{k-1}^{\Delta})+1.
\qed
$$
\end{lmm}

\section{The spectral sequences.}


Because the following notations were already given in
\cite[section 3]{AKY1}, we only explain the important ones.
\begin{dfn}\label{Def: 3.1}
{\rm
Fix a based algebraic map $g\in\Alg_d^*(\RP^{m-1},\RP^n)$ of degree $d$ together with a representation $(g_0,\ldots,g_n)\in A_d(m-1,n)$. Note that 
$A_d(m,n;g)$ is an open subspace of the affine space $A_d^*$
as defined in (\ref{ad}). 
\par\vspace{2mm}\par
(i)
 Let $N_d^*=\dim A_d^*=(n+1)\binom{m+d-1}{m}$, and 
define $\Sigma_d^*\subset A_d^*$ as the \emph{discriminant} of $A_d(m,n;g)$ in $A_d^*$ by
$$\Sigma_d^*=A_d^*\setminus A_d(m,n;g).$$
In other words, $\Sigma_d^*$ consists of the $(n+1)$-tuples of polynomials in $A_d^*$ which have at least one nontrivial common real root.
\par
(ii)
Let  $Z_d^{*}\subset \Sigma_d^{*}\times \R^m$
denote 
{\it the tautological normalization} of 
 $\Sigma_d^*$
consisting of all pairs 
$({\bf f},{\bf x})=((f_0,\ldots ,f_n),
(x_0,\ldots ,x_{m-1}))\in \Sigma_d^{*}\times\R^m$
such that the polynomials $f_0,\ldots ,f_n$ have a non-trivial common real root
$({\bf x},1)=(x_0,\ldots ,x_{m-1},1)$.
Projection on the first factor  gives a surjective map
$\pi_d^{\p} :Z_d^{*}\to\Sigma_d^{*}.$

}
\end{dfn}

Our goal in this section is to construct, by means of the
non-degenerate simplicial resolution  of the discriminant, a spectral sequence converging to the homology of the space
$A_d(m,n;g)$.
\begin{dfn}\label{non-degenerate simp.}
{\rm
Let 
$(\SZ,{\pi_d}^{\Delta}:\SZ\to\Sigma_d^*)$ 
denote the non-degenerate simplicial resolution of the surjective map
$\pi_d^{\p}:Z_d^*\to \Sigma_d$ 
with the natural increasing filtration
$$
\SZ_0=\emptyset
\subset \SZ_1\subset 
\SZ_2\subset \cdots
\subset 
\SZ=\bigcup_{k= 0}^{\infty}\SZ_k.
$$
}
\end{dfn}


By Lemma \ref{lemma: simp} the map
$^{\p}{\pi}_d^{\Delta}:
\SZ\stackrel{\simeq}{\rightarrow}\Sigma_d^*$
is a homotopy equivalence, and it
extends to  a homotopy equivalence
$^{\p}{\pi_{d+}^{\Delta}}:\SZ_+\stackrel{\simeq}{\rightarrow}{\Sigma_{d+}^*},$
where $X_+$ denotes the one-point compactification of a
locally compact space $X$.
\par
Since
${\cal X}^{\Delta ,d}_+/{\SZ_{r-1}}_+
\cong (\SZ_r\setminus \SZ_{r-1})_+$,
we have a spectral sequence 
$$\big\{E_t^{r,s}(d),
d_t:E_t^{r,s}(d)\to E_t^{r+t,s+1-t}(d)
\big\}
\Rightarrow
H^{r+s}_c(\Sigma_d^*,\Z),
$$
where
$E_1^{r,s}(d):=H^{r+s}_c(\SZ_r\setminus\SZ_{r-1},\Z)$ and
$H_c^k(X,\Z)$ denotes the cohomology group with compact supports given by 
$
H_c^k(X,\Z):= H^k(X_+,\Z).
$
\par
By the Alexander duality  there is a natural
isomorphism
\begin{equation}\label{Al}
H_k(A_d(m,n;g),\Z)\cong
H_c^{N_d^*-k-1}(\Sigma_d^*,\Z)
\quad
\mbox{for }1\leq k\leq N_d^*-2.
\end{equation}
Using (\ref{Al}) and
reindexing we obtain a
spectral sequence
\begin{eqnarray}\label{SS}
&&\big\{\E^t_{r,s}(d), \tilde{d}^t:\E^t_{r,s}(d)\to \E^t_{r+t,s+t-1}(d)\big\}
\Rightarrow H_{s-r}(A_d(m,n;g),\Z)
\end{eqnarray}
if $s-r\leq N_d^*-2$,
where
$\E^1_{r,s}(d)=
H^{N_d^*+r-s-1}_c(\SZ_r\setminus\SZ_{r-1},\Z).$
\par\vspace{2mm}\par
For a connected space $X$, 
let $F(X,r)$ denote the configuration space  of distinct $r$ points in $X$.
The symmetric group $S_r$ of $r$ letters acts on $F(X,r)$ freely by permuting 
coordinates, and let $C_r(X)$ be the configuration space of unordered 
$r$-distinct  points in $X$ given by 
$C_r(X)=F(X,r)/S_r$.

\begin{lmm}\label{lemma: vector bundle*}
If $1\leq r\leq d+1$ and $m\geq 2$,
$\SZ_r\setminus\SZ_{r-1}$
is homeomorphic to the total space of a real open disk
bundle $\xi_{d,r}$ over $C_r(\R^m)$ with rank 
$l_{d,r}^*:=N_d^*-nr-1$.
\end{lmm}
\begin{proof}
The argument is exactly analogous to the one in the proof of  
\cite[Prop. 3.1]{Mo3}. Namely, an element of $\SZ_r\setminus\SZ_{r-1}$ is represented by $((f_0,\cdots ,f_n),u)$, where the $f_i$ are polynomials in $\Sigma_d^*$ and $u$ is an element of the span of the images of distinct $r$ points $x_1,\cdots, x_n$ at which these polynomials vanish, under a suitable embedding. 
There is a well defined map $\pi :{\cal X}^{\Delta,d}_r\setminus
{\cal X}^{\Delta,d}_{r-1}\to C_r(\R^m)$ given by
$((f_0,\cdots ,f_n),u) \mapsto \{x_1,\dots, x_r\}$. (The points $x_i$ are uniquely determined by $u$ by the construction of the non-degenerate simplicial resolution.)
Moreover, the condition that the polynomials vanish at $r$ points gives $r$ independent linear equations for the coefficients of the polynomials, as long as $r \le d+1$.
Hence, $\pi$ gives the open disk bundle over $C_r(\R^m)$ with rank $l^*_{d,r}
=N_d^*-nr-1$.
\end{proof}
\begin{rem}
{\rm
If $m=1$, the assertion of Lemma \ref{lemma: vector bundle*} holds only for 
$1\leq r\leq d$.
In fact,
if $m=1$ and $r=d+1$, because
 an equation in one variable of degree $d$ can have at most $d$ real roots,
the condition that the polynomials vanish at $r$ points does not give  $r$ independent linear equations for the coefficients of the polynomials.}
\end{rem}
\begin{lmm}\label{lemma: E11}
If $1\leq r\leq  d+1$, there is a natural isomorphism
$$
\E^1_{r,s}(d)\cong
H_{s-(n-m+1)r}(C_r(\R^m),(\pm \Z)^{\otimes (n-m)}).
$$
Here the meaning of  $(\pm \Z)^{\otimes (n-m)}$  is the same as in {\rm \cite{Va}}.
\end{lmm}
\par\noindent{\it Proof. }
Suppose that $1\leq r\leq d+1$.
By Lemma \ref{lemma: vector bundle*}, there is a
homeomorphism
$(\SZ_r\setminus\SZ_{r-1})_+\cong T(\xi_{d,r}),$
where $T(\xi_{d,r})$ denotes the Thom complex of
$\xi_{d,r}$.
Since 
$N_d^*+r-s-1-l_{d,r}^*=(n+1)r-s$
and $rm-\{(n+1)r-s\}=s-(n-m+1)r$,
by using the Thom isomorphism and Poincar\'e duality,
we obtain a natural isomorphism 
$$
E^1_{r,s}
\cong H^{N_d^*+r-s-1}(T(\xi_{d,r}),\Z)
\cong
H_{s-(n-m+1)r}(C_r(\R^m), (\pm \Z)^{\otimes (n-m)}).
\qed
$$

\section{Spectral sequences induced from truncated resolutions.}

In this section, we prove a  key result (Theorem \ref{thm: sd}) about the homology stability of \lq\lq stabilization 
maps $s_d:A_d(m,n;g)\to A_{d+2}(m,n;g)$.

\begin{dfn}
{\rm
Let $X^{\Delta}$ denote the (after $(d+1)$-th term) truncated simplicial resolution 
of $\SZ$ with its natural filtration
$$
\emptyset =X^{\Delta}_0\subset X^{\Delta}_1\subset
\cdots \subset  X^{\Delta}_{d+1}\subset X^{\Delta}_{d+2}=
X^{\Delta}_{d+3}=
\cdots
 =X^{\Delta},
$$
where $X^{\Delta}_k=\SZ_k$ if $k\leq d+1$ and
$X^{\Delta}_k=X^{\Delta}$ if $k\geq d+2$.
}
\end{dfn}

\begin{rem}
{\rm
Note that our notation  
$X^{\Delta}$ conflicts with that of  \cite{Mo3} and Definition 2.3, because usually ${\cal X}^{\Delta}$ denotes the non-degenerate simplicial resolution. 
}
\end{rem}
By Lemma \ref{Lemma: truncated} that 
there is a homotopy equivalence
$\pi^{\Delta}:X^{\Delta} \stackrel{\simeq}{\rightarrow}
\Sigma_d^*$.
Hence, analogously to (\ref{Al}) and (\ref{SS}),
we obtain a spectral sequence

\begin{eqnarray}\label{SSS}
&&\big\{E^t_{r,s}, d^t:E^t_{r,s}\to E^t_{r+t,s+t-1}\big\}
\Rightarrow H_{s-r}(A_d(m,n;g),\Z)
\end{eqnarray}
if $s-r\leq N_d^*-2$,
where
$E^1_{r,s}=
H^{N_d^*+r-s-1}_c(X^{\Delta}_r\setminus X^{\Delta}_{r-1},\Z).$

\begin{lmm}\label{lemma: E1}
If $1\leq r\leq  d+1$, there is a natural isomorphism
$$
E^1_{r,s}\cong
H_{s-(n-m+1)r}(C_r(\R^m),(\pm \Z)^{\otimes (n-m)}).
$$
\end{lmm}
\begin{proof}
Since 
$X^{\Delta}_r\setminus X^{\Delta}_{r-1}
=\SZ_r\setminus\SZ_{r-1}$ for $1\leq r\leq d+1$, the assertion follows from
Lemma \ref{lemma: E11}. 
\end{proof}

\begin{lmm}\label{lemma: range}
$\I$ $E^1_{r,s}=0$ if $r<0$, or if $r\geq d+3$, or if $r=0$ and $s\not= N_d^*-1$.
\par
$\II$ If $1\leq r\leq d+1$,
$E^1_{r,s}=0$ for $s\leq (n-m+1)r-1$.
\par
$\III$ If $r=d+2$,
$E^1_{d+2,s}=0$ for $s\leq (n-m+1)(d+1)-1$.
\end{lmm}
\begin{proof}
(i)
Since $\XD_0=\emptyset$ and $X^{\Delta}=\XD_k$ for $k\geq d+3$,
(i) is trivial.
\par
(ii)
If $1\leq r\leq d+1$, by Lemma \ref{lemma: E1}, there is an
isomorphism
$$
E^1_{r,s}\cong
H_{s-(n-m+1)r}(C_r(\R^m),(\pm \Z)^{\otimes (n-m)}).
$$
Because $s-(n-m+1)r<0$ if and only if
$s\leq (n-m+1)r-1$, (ii) follows.
\par
(iii)
An easy computation shows that
$$
\dim (
\SZ_{r}\setminus \SZ_{r-1})=
N_d^*-r(n+1)+r-1+rm=
N_d^*-r(n-m)-1.
$$
By Lemma \ref{Lemma: dimension}, 
$\dim (\XD_{d+2}\setminus \XD_{d+1})=
\dim (
\SZ_{d+1}\setminus \SZ_{d})+1.$
So
$
\dim (\XD_{d+2}\setminus \XD_{d+1})
=N_d^*-(d+1)(n-m).$
Since
$
E^1_{d+2,s}=H_c^{N_d^*+d+1-s}(\XD_{d+2}\setminus \XD_{d+1},\Z)
$
and $N_d^*+d+1-s>N_d^*-(n-m)(d+1)$ $\Leftrightarrow$
$s\leq (n-m+1)(d+1)-1$, we see that
$E^1_{d+2,s}(d)=0$ for $s\leq (n-m+1)(d+1)-1.$
\end{proof}
Similarly,
let $Y^{\Delta}$ denote the (after $(d+1)$-th term) truncated simplicial resolution 
of $\SZd$ with its natural filtration
$$
\emptyset =
Y^{\Delta}_0\subset Y^{\Delta}_1\subset
\cdots \subset  Y^{\Delta}_{d+1}\subset Y^{\Delta}_{d+2}=
Y^{\Delta}_{d+3}=
\cdots
 =Y^{\Delta},
$$
where $Y^{\Delta}_k=\SZd_k$ if $k\leq d+1$ and
$Y^{\Delta}_k=Y^{\Delta}$ if $k\geq d+2$.
\par
By Lemma \ref{Lemma: truncated} that 
there is a homotopy equivalence
$^{\p}\pi^{\Delta}:Y^{\Delta} \stackrel{\simeq}{\rightarrow}
\Sigma_{d+2}^*$.
Hence, by using the same method as above,
we obtain a spectral sequence
\begin{eqnarray}\label{SSSS}
&&\big\{\ ^{\p}E^t_{r,s},\  ^{\p}d^t:E^t_{r,s}\to 
\ ^{\p}E^t_{r+t,s+t-1}\big\}
\Rightarrow H_{s-r}(A_{d+2}(m,n;g),\Z)
\end{eqnarray}
if $s-r\leq N_{d+2}^*-2$,
where
$^{\p}E^1_{r,s}=
H^{N_{d+2}^*+r-s-1}_c(Y^{\Delta}_r\setminus Y^{\Delta}_{r-1},\Z).$
\par
Applying again the same argument  we obtain the following
two results.

\begin{lmm}\label{lemma: E1*}
If $1\leq r\leq  d+1$, there is a natural isomorphism
$$
^{\p}E^1_{r,s}\cong
H_{s-(n-m+1)r}(C_r(\R^m),(\pm \Z)^{\otimes (n-m)}).
\qed
$$
\end{lmm}

\begin{lmm}\label{lemma: range*}
$\I$
$^{\p}E^1_{r,s}=0$ if $r<0$, or if $r\geq d+3$, or if
$r=0$ and $s\not= N_{d+2}^*-1$.
\par
$\II$ If $1\leq r\leq d+1$,
$^{\p}E^1_{r,s}=0$ for $s\leq (n-m+1)r-1$.
\par
$\III$ If $r=d+2$,
$^{\p}E^1_{d+2,s}=0$ for $s\leq (n-m+1)(d+1)-1$.
\qed
\end{lmm}
\begin{dfn}\label{Def: 4.6}
{\rm
Let $g\in \Alg_d^*(\RP^{m-1},\RP^n)$ be a fixed algebraic map,
and let $(g_0,\cdots ,g_n)\in A_d(m-1,n)$ be its fixed representative.
If we set $\tilde{g}=\sum_{k=0}^mz_k^2,$ we can see that the tuple
$(\tilde{g}g_0,\cdots ,\tilde{g}g_n)$ can also be chosen as a representative of the
map $g\in \Alg_{d+2}^*(\RP^{m-1},\RP^n)$.
So one can define a stabilization map
\begin{equation}\label{sd}
s_d:A_d(m,n;g)\to A_{d+2}(m,n;g)
\end{equation}
by
$
s_d(f_0,\cdots ,f_n)=(f_0\tilde{g},\cdots ,f_n\tilde{g})$
for $(f_0,\cdots ,f_n)\in A_d(m,n;g).$
\par
Because there is a commutative diagram 
$$
\begin{CD}
A_d(m,n;g) @>s_d>> A_{d+2}(m,n;g)
\\
@V{i_d^{\p}}VV @V{i_{d+2}^{\p}}VV
\\
F(m,n;g) @>=>> F(m,n;g)
\end{CD}
$$
it induces a map
\begin{equation}\label{sdinfty}
s_{d,\infty}=\lim_{k\to\infty} s_{d+2k}:
A_{d,\infty}(m,n;g)
\to F(m,n;g)\simeq\Omega^mS^n,
\end{equation}
where 
$A_{d,\infty}(m,n;g)$ denotes
the colimit 
$\dis \lim_{k\to \infty}A_{d+2k}(m,n;g)$
induced from the stabilization maps 
$s_{d+2k}$'s $(k\geq 0)$.}
\end{dfn}
\begin{thm}\label{thm: AKY1-stable}
If $2\leq m<n$, the map
$\dis s_{d,\infty}:A_{d,\infty}(m,n;g)
\stackrel{\simeq}{\rightarrow} \Omega^mS^n$
is a homotopy equivalence if $m+2\leq n$ and is a
homology equivalence if $m+1=n$.
\end{thm}
\begin{proof}
This easily follows from Theorem \ref{thm: AKY1-I}.
\end{proof}

Now consider the stabilization map
$s_d:A_d(m,n;g)\to A_{d+2}(m,n;g)$.
This map naturally extends to the map 
$\tilde{s}_d:\Sigma_d^*\to \Sigma_{d+2}^*$ by the multiplication by
$\tilde{g}$,
$$
\tilde{s}_d(f_0,\cdots ,f_n)=(f_0\tilde{g},\cdots ,f_n\tilde{g})
\quad
\mbox{for }(f_0,\cdots ,f_n)\in \Sigma_d^*.
$$ 
It also naturally extends to the filtration preserving map
$\hat{s}_d:\SZ\to \SZd$ between non-degenerate resolutions.
Hence, it naturally extends to the filtration preserving map
$\hat{s}_d:\XD  \to Y^{\Delta}$, and induces a homomorphism 
of spectral sequences
\begin{equation}
\{\theta^t_{r,s}:E^t_{r,s}\to \ ^{\p}E^t_{r,s}\}.
\end{equation}

\begin{lmm}\label{lemma: Thom}
If $1\leq r\leq d+1$,
$\theta^1_{r,s}:E^1_{r,s}\to \ ^{\p}E^1_{r,s}$
is an isomorphism for any $s$.
\end{lmm}
\begin{proof}
Suppose that $1\leq r\leq d+1$.
Then
it follows from the proof of Lemma \ref{lemma: vector bundle*}
that there is a commutative diagram of open disk bundles
$$
\begin{CD}
\SZ_r\setminus \SZ_{r-1} @>\pi>> 
C_r(\R^m)
\\
@V{\hat{s}_d}VV  \Vert @.
\\
\SZd_r\setminus \SZd_{r-1} @>{\pi}>> C_r(\R^m)
\end{CD}
$$
Since $X^{\Delta}_r\setminus X^{\Delta}_{r-1}=
\SZ_r\setminus \SZ_{r-1}$ and
$Y^{\Delta}_r\setminus Y^{\Delta}_{r-1}=
\SZd_r\setminus \SZd_{r-1}$,
by  Lemma \ref{lemma: E1}, Lemma \ref{lemma: E1*}
and the naturality of the Thom isomorphism,
we have a commutative diagram
\begin{equation*}\label{Thom}
\begin{CD}
E^1_{r,s}
@>T>\cong> H_{s-r(n-m+1)}(C_r(\R^m),(\pm \Z)^{\otimes (n-m)})
\\
@V{\theta}^1_{r,s}VV  \Vert @.
\\
^{\p}E^1_{r,s}
@>T>\cong> H_{s-r(n-m+1)}(C_r(\R^m),(\pm \Z)^{\otimes (n-m)} )
\end{CD}
\end{equation*}
where $T$ denotes the Thom isomorphism.
Hence, 
$\theta^1_{r,s}$ is an isomorphism.
\end{proof}

\begin{thm}\label{thm: sd}
If $2\leq m<n$,
$s_d:A_d(m,n;g) \to A_{d+2}(m,n;g)$ is a homology equivalence through
dimension $D(d;m,n)-1$.
\end{thm}
\begin{proof}
Recall that we have 
two spectral sequences
$$
\begin{cases}
\{E^{t}_{t,s},d^t:E^t_{r,s}\to \ E^t_{r+t,s+t-1}\}\ & 
\Rightarrow \quad H_{s-r}(A_d(m,n;g),\Z),
\\
\{\ ^{\p}E^{t}_{t,s},\ ^{\p}d^t:\ ^{\p}E^t_{r,s}
\to \ ^{\p}E^t_{r+t,s+t-1}\}  & 
\Rightarrow \quad H_{s-r}(A_{d+2}(m,n;g),\Z),
\end{cases}
$$
and a homomorphism 
$\{\theta^t_{r,s}:E^t_{r,s}\to \ ^{\p}E^t_{r,s}\}$
of spectral sequences.
\par
Now we shall consider the maximal positive integer $D_{max}$ such that
$$
D_{max}=\max\{D\in\Z:\theta^{\infty}_{r,s}
\mbox{ is always an isomorphism as long as }s-r\leq D\}.
$$
By Lemmas \ref{lemma: range} and  \ref{lemma: range*}, we see that
$E^1_{r,s}=\ ^{\p}E^1_{r,s}=0$ if
$r<0$, or if $r>d+2$, or if $r=d+2$ with $s\leq (n-m+1)(d+1)-1$.
Since $(n-m+1)(d+1)-(d+2)=D(d;m,n)-1$,  we easily see that:
\begin{enumerate}
\item[$(*)_1$]
If $r<0$ or $r\geq d+2$,
$\theta^{\infty}_{r,s}$ is an isomorphism for all $(r,s)$ such that
$s-r\leq D(d;m,n)-1$.
\end{enumerate}
Next, we assume that $0\leq r\leq d+1$, and investigate 
the condition that $\theta^{\infty}_{r,s}$ is an isomorphism.
Note that the group $E^1_{r_1,s_1}$ is not known for
$(r_1,s_1)\in{\cal S}_1=\{(d+2,s)\in\Z^2:s\geq (n-m)(d+1)\}$.
However,
by considering the differentials
$d^1:E^1_{r,s}\to E^{1}_{r+1,s}$
and
$\ ^{\p}d^1:\ ^{\p}E^1_{r,s}\to \ ^{\p}E^{1}_{r+1,s}$
and applying Lemma \ref{lemma: Thom}, we see that
$\theta^2_{r,s}$ is an isomorphism if
$(r,s)\notin {\cal S}_2$, where
\begin{eqnarray*}
{\cal S}_2&=&
\{(r_1,s_1)\in\Z^2:(r_1+1,s_1)\in {\cal S}_1\}
\\
&=&
\{(d+1,s_1)\in \Z^2:s_1\geq (n-m+1)(d+1)\}.
\end{eqnarray*}
A similar argument for the differentials $d^2$ and $^{\p}d^2$ shows that
$\theta^3_{r,s}$ is an isomorphism if
$(r,s)\notin {\cal S}_3=\{(r_1,s_1)\in\Z^2:(r_1+2,s_1+1)\in {\cal S}_1\cup
{\cal S}_2\}.$
\par
Continuing in the same fashion,
considering the differentials
$d^k:E^k_{r,s}\to E^{k}_{r+t,s+t-1}$
and
$\ ^{\p}d^k:\ ^{\p}E^k_{r,s}\to \ ^{\p}E^{k}_{r+t,s+t-1},$
and applying 
Lemma \ref{lemma: Thom}, we  easily see that $\theta^{\infty}_{r,s}$ is an isomorphism
if $\dis (r,s)\notin {\cal S}:=\bigcup_{t\geq 1}{\cal S}_t
=\bigcup_{t\geq 1}A_t$,
where  $A_t$ denotes the set given by
$$
A_t=
\left\{
\begin{array}{c|l}
 &\mbox{There are  integers }k_1,k_2,\cdots ,k_t
\mbox{ such that},
\\
(r_1,s_1)\in \Z^2 &\  1\leq k_1<k_2<\cdots <k_t,\ 
r_1+\sum_{l=1}^tk_l=d+2,
\\
& \ s_1+\sum_{l=1}^t(k_l-1)\geq (n-m+1)(d+1)
\end{array}
\right\}.
$$
If $\dis A_t\not= \emptyset$, it is easy to see that
\begin{eqnarray*}
a(t)&=&\min \{s-r:(r,s)\in A_t\}=
(n-m+1)(d+1)-(d+2)+t
\\
&=&(n-m)(d+1)-1+t
=D(d;m,n)+t.
\end{eqnarray*}
Hence, 
$\min \{a(t):t\geq 1,A_t\not=\emptyset\}=D(d;m,n)+1,$
and we have the following:
\begin{enumerate}
\item[$(*)_2$]
If $0\leq r\leq d+1$,
$\theta^{\infty}_{r,s}$ is  an isomorphism for any $(r,s)$ such that
$s-r\leq  D(d;m,n).$
\end{enumerate}
Then, by $(*)_1$ and $(*)_2$, we see that
$\theta^{\infty}_{r,s}:E^{\infty}_{r,s}\stackrel{\cong}{\rightarrow}
\ ^{\p}E^{\infty}_{r,s}$ is an isomorphism for any $(r,s)$
such that $s-r\leq D(d;m,n)-1$.
Thus, by using the Comparison Theorem for spectral sequences,
$s_d$ is a homology equivalence through dimension
$D(d;m,n)-1$.
\end{proof}

\section{The proof of the main result.}
In this section, we prove our main result (Theorem \ref{thm: KY4-I}). Note that we could simply deduce a weaker result (with $D(d;m,n)-1$ in place of $D(d;m,n)$) by simply combing  Theorem \ref{thm: AKY1-stable} with Theorem \ref{thm: sd}.  However, we can do better by considering a \lq\lq stable\rq\rq non-degenerate resolution in the manner of \cite{Va}. 

\begin{dfn}
{\rm
Let 
\begin{equation}\label{j'}
j_d^{\p}:A_d(m,n;g)\to A_{d,\infty}(m,n;g)=
\lim_{k\to\infty}A_{d+2k}(m,n;g)
\end{equation}
 be the natural map.}
\end{dfn}
Recall that $\SZ$ is a non-degenerate simplicial resolution of
$\pi^{\p}_d:Z^*_d\to\Sigma_d^*$ and it
can be defined by using the family of embeddings
${\cal E}_d=\{\tilde{i}_{r,d}:Z^*_d\to \R^{N_{r}}\}_{r\geq 1}$
satisfying  the condition (\ref{lemma: simp}$)_r$
as explained in  Remark \ref{Remark: non-degenerate}.
The stabilization map $\tilde{s}_d:\Sigma_d^*\to \Sigma_{d+2}^*$ naturally
extends to the map $\tilde{s}_d:Z_d^*\to Z^*_{d+2}$ by  using the multiplication by
$\tilde{g}$.
Then it is easy to see that we can choose the families $\{{\cal E}_d\}_{d\geq 1}$
of the embeddings  
satisfying the two conditions (\ref{lemma: simp}$)_r$ and 
\begin{equation}\label{condition}
\tilde{i}_{r,d+2}\circ\tilde{s}_d=\tilde{i}_{r,d}
\quad
\mbox{ for each pair of positive integers }(d,r).
\end{equation}
Since
$s_{d+2k}$ 
induces the filtration preserving 
map
$\hat{s}_{d+2k}:{\cal X}^{\Delta ,d+2k}\to
{\cal X}^{\Delta ,d+2k+2}$
between non-degenerate simplicial resolutions,
it also gives the filtration preserving map
$
\hat{s}_{d,\infty}:
{\cal X}^{\Delta ,d}\to
{\cal X}^{\Delta}_{d,\infty},
$
where ${\cal X}^{\Delta}_{d,\infty}$ denotes the colimit 
$\dis \lim_{k\to\infty}{\cal X}^{\Delta ,d+2k}$
induced from the maps
$\hat{s}_{d+2k}$'s.
\par
It follows from the method due to Vassiliev
(\cite[\S 5 of Chap. III]{Va}, \cite[4.3]{AKY1}),
(\ref{condition}) and Lemma \ref{lemma: E11} that 
we may regard 
${\cal X}^{\Delta}_{d,\infty}$ as a  non-degenerate simplicial
resolution of the discriminant 
of $A_{d,\infty}(m,n;g)$, and
that there is a spectral sequence
$
\{\E^t_{r,s},\tilde{d}^t:
\E^t_{r,s}\to \E^t_{r+t,s+t-1}\}
\ \Rightarrow H_{s-r}(A_{d,\infty}(m,n;g),\Z),
$ such that,
\begin{equation}
\label{Einfty}
\dis
\E^1_{r,s}=
H_{s-(n-m+1)r}(C_r(\R^m),(\pm \Z)^{\otimes (n-m)})
\quad\mbox{ for any }r\geq 1.
\end{equation} 
Since $\hat{s}_{d,\infty}$ is a
filtration preserving map between  non-degenerate simplicial
resolutions, it induces a homomorphism of spectral sequences,
\begin{equation}\label{spectralseqhomo}
\{\tilde{\theta}^t_{r,s}:
\E^t_{r,s}(d)\to \E^t_{r,s}\}.
\end{equation}

\begin{lmm}
\label{lemma: Einftyspectal}
$\E^1_{r,s}=\E^{\infty}_{r,s}$
for any $(r,s)$, and
$\tilde{d}^t:\E^t_{r,s}\to \E^t_{r+t,s+t-1}$
is trivial for any $t\geq 1$.
Moreover,
if $k\geq 1$, the extension problem for
the graded group
$Gr(H_k(A_{d,\infty}(m,n;g),\Z))$ is trivial and
there is an isomorphism
$$
H_k(A_{d,\infty}(m,n;g),\Z)\cong
\bigoplus_{r=1}^{\infty}\E^1_{r,r+k}=
\bigoplus_{r=1}^{\infty}H_{k-(n-m)r}(C_r(\R^m),
(\pm \Z)^{\otimes (n-m)}).
$$
\end{lmm}
\begin{proof}
First, note that
by \cite{Va} (cf. \cite{CLM}, \cite{Sn})
there is an isomorphism
\begin{equation}\label{Vas}
H_k(\Omega^mS^n,\Z)\cong
\bigoplus_{r=1}^{\infty}
H_{k-(n-m)r}(C_r(\R^m),(\pm \Z)^{\otimes (n-m)})
\quad\mbox{ for }k\geq 1.
\end{equation}
Hence, by Theorem \ref{thm: AKY1-stable} for each $k\geq 1$,
there is an isomorphism
\begin{equation*}\label{Ainfty}
 H_k(A_{d,\infty}(m,n;g),\Z)
\cong 
\bigoplus_{r=1}^{\infty}
H_{k-r(n-m)}(C_r(\R^m),(\pm \Z)^{\otimes (n-m)})
=\bigoplus_{r=1}^{\infty}\E^1_{r,r+k}.
\end{equation*}
So, $\E^1_{r,s}=\E^{\infty}_{r,s}$ if $s-r\geq 1$.
If $s-r\leq 0$ then, due to dimensional reasons,
$$
\E^1_{r,s}=\E^{\infty}_{r,s}=
\begin{cases}
0 & \mbox{if }(r,s)\not= (0,0),
\\
\Z & \mbox{if }(r,s)=(0,0).
\end{cases}
$$
Hence,
$\E^1_{r,s}=\E^{\infty}_{r,s}$ for any $(r,s)$,  so that
$\tilde{d}^t=0$ for any $t\geq 1$.
Finally, by from 
 (\ref{Einfty}), we easily see that the extension problem of the graded group
 is trivial.
\end{proof}

\begin{lmm}\label{lemma: keylemma}
\begin{enumerate}
\item[$\I$]
If $1\leq r\leq d+1$,
$\tilde{\theta}^1_{r,s}:\E^1_{r,s}(d)\stackrel{\cong}{\longrightarrow} \E^1_{r,s}$
is an isomorphism for any $s$.
\item[$\II$]
If $s-r<D(d;m,n)$,
$\tilde{d}^t:\E^t_{r,s}(d)\to \E^t_{r+t,s+t-1}(d)$ is
trivial for any $t\geq 1$.
\item[$\III$]
If $1\leq r\leq d$ and $s-r =D(d;m,n)$,
$\tilde{\theta}^{\infty}_{r,s}:
\E^{\infty}_{r,s}(d) \to \E^{\infty}_{r,s}$ is 
an isomorphism.
\end{enumerate}
\end{lmm}
\begin{proof}
(i) The assertion easily follows from
the proof of Lemma \ref{lemma: Thom}.
\par
(ii)
By Theorem \ref{thm: AKY1-stable}, (\ref{Vas})  and
for dimensional reasons,
there is an isomorphism
$$
H_k(A_d(m,n;g),\Z)\cong
\bigoplus_{r=1}^{d}H_{k-(n-m)r}(C_r(\R^m),(\pm \Z)^{\otimes
(n-m)})
\quad
\mbox{if }k<D(d;m,n).
$$
Hence, by using the spectral sequence
$
\{\E^t_{r,s}(d),\tilde{d}^t:
\E^t_{r,s}(d)\to \E^t_{r+t,s+t-1}(d)\},$
we see that
$\E^1_{r,s}(d)=\E^{\infty}_{r,s}(d)$ if $s-r< D(d;m,n)$,
and (ii) follows.
\par
(iii)
Assume that  $s=r+D(d;m,n)$.
Since
$(s+t-1)-(r+t)=D(d;m,n)-1$, by  (ii)
$\E^1_{r+t,s+t-1}(d)=\E^{\infty}_{r+t,s+t-1}(d)$ for any $t\geq 1$.
Hence,
$\tilde{d}^t:\E^t_{r,s}(d)\to \E^{t}_{r+t,s+t-1}(d)$
is trivial for any $t\geq 1$, and
there is a natural epimorphism
$\hat{\pi} :\E^1_{r,s}(d)\to \E^{\infty}_{r,s}(d)$.
By Lemma \ref{lemma: Einftyspectal}, 
we have the following commutative diagram
\begin{equation}\label{Einftydiagram}
\begin{CD}
\E^1_{r,s}(d) @>{\tilde{\theta}^1_{r,s}}>{\cong}> \E^1_{r,s}
\\
@V{\hat{\pi}}V{\mbox{\tiny epic.}}V  \Vert @.
\\
\E^{\infty}_{r,s}(d) @>{\tilde{\theta}^{\infty}_{r,s}}>> 
\E^{\infty}_{r,s}
\end{CD}
\end{equation}
Since $1\leq r\leq d$,
$\tilde{\theta}^1_{r,s}$ is an isomorphism
by (i).
Hence,  easy diagram chasing shows
that
$\tilde{\theta}^{\infty}_{r,s}$ is an isomorphism.
\end{proof}
\begin{crl}\label{cor: collapse}
If $1\leq r\leq d$ and $s-r\leq D(d;m,n)$,
$\E^1_{r,s}(d)$ collapses at the $\E^1(d)$ term.
Hence, if $1\leq r\leq d$ and $s-r\leq D(d;m,n)$,
there is an isomorphism
$
\E^1_{r,s}(d)=\E^{\infty}_{r,s}(d)
\cong
H_{s-(n-m+1)r}(C_r(\R^m),(\pm \Z)^{\otimes (n-m)}).
$
\end{crl}
\begin{proof}
Since the proof is analogous, we give the proof only when
$s-r=D(d;m,n)$.
If $s-r=D(d;m,n)$, by (\ref{Einftydiagram}) and 
Lemma \ref{lemma: keylemma}, $\hat{\pi}$ is an isomorphism.
Hence, $\E^1_{r,s}(d)=\E^{\infty}_{r,s}(d)
\cong
H_{s-(n-m+1)r}(C_r(\R^m),(\pm \Z)^{\otimes (n-m)}).$
\end{proof}
\begin{thm}\label{thm: jd}
If $2\leq m<n$, the map
$\dis j_d^{\p}:A_d(m,n;g)\to A_{d,\infty}(m,n;g)$ is a 
homology equivalence up to dimension $D(d;m,n)$.
\end{thm}
\begin{proof}
It follows from Theorem \ref{thm: sd} that
the map $j_d^{\p}$ is a homology equivalence through dimension
$D(d;m,n)-1$.
So it remains to show that
$j_d^{\p}$ induces an epimorphism on the homology
$H_k(\ ,\Z)$ for $k=D(d;m,n)$.
However,
since 
$j_d^{\p}$ 
induces the homomorphism
$\{\tilde{\theta}^t_{r,s}:
\E^t_{r,s}(d)\to \E^t_{r,s}\}$ of spectral sequences,
it suffices to prove that
$\tilde{\theta}^{\infty}_{r,s}:
\E^{\infty}_{r,s}(d)\to \E^{\infty}_{r,s}$
is an epimorphism if $s-r=D(d;m,n)$.
If $r\leq 0$, $\E^{\infty}_{r,s}=0$. So in this case the assertion is
trivial. If $r\geq d+1$, 
$s-(n-m+1)r=(n-m)(d+1-r)-1<0$. 
Hence,  
$\E^{1}_{r,s}
=H_{s-(n-m+1)r}(C_r(\R^m),(\pm \Z)^{\otimes (n-m)})=0$ 
if $r\geq d+1$.
So $\E^{\infty}_{r,s}=0$ if $r\geq d+1$ and the assertion is also true in this case.
Finally,
if $1\leq r\leq d$, then, by Lemma \ref{lemma: keylemma},
$\tilde{\theta}^{\infty}_{r,s}$ is an isomorphism and
the result follows. 
\end{proof}

\begin{proof}[Proof of Theorem \ref{thm: KY4-I}]
Note that the map $i_d^{\p}$ coincides the composite of maps
$$
A_d(m,n;g) \stackrel{j_d^{\p}}{\longrightarrow}
A_{d,\infty}(m,n;g)
\stackrel{s_{d,\infty}}{\longrightarrow}\Omega^mS^n.
$$
It follows from Theorem \ref{thm: AKY1-stable} that
$s_{d,\infty}$ is a homology equivalence.
Since $j_d^{\p}$ is a homology equivalence up to dimension
$D(d;m,n)$ by Theorem \ref{thm: jd},
the map $i_d^{\p}$ is so.
If $m+2\leq n$,
$A_d(m,n;g)$ and $\Omega^mS^n$ are simply connected
by \cite[Fact 3.2]{AKY1} and  $i_d^{\p}$ is a
homotopy equivalence up to dimension
$D(d;m,n)$.
\end{proof}

\par\vspace{2mm}\par
\noindent{\bf Acknowledgements. }
Both authors should like to take this opportunity to thank  Professor Jacob Mostovoy 
for his valuable suggestions.








\end{document}